\documentclass{amsart}
\usepackage{amsmath,amscd,amssymb,enumerate,verbatim,amsthm}
\usepackage[T1]{fontenc}
\usepackage{parskip}

\newcommand{\mC}{\mathcal{C}}

\newcommand{\mM}{\mathcal{M}}

\newcommand{\CC}{\mathbf{C}}
\newcommand{\RR}{\mathbf{R}}

\newcommand{\ZZ}{\mathbf{Z}}

\newcommand{\id}{\mathrm{Id}}

\newcommand{\ev}{\mathrm{ev}}

\newcommand{\Hom}{\mathrm{Hom}}

\newcommand{\MT}{\mathrm{MT}}
\newcommand{\OP}{\operatorname}

\newtheorem{thm}{Theorem}
\newtheorem{dfn}[thm]{Definition}
\newtheorem{lma}[thm]{Lemma}
\newtheorem{prp}[thm]{Proposition}

\newtheorem{thmmain}{Theorem}
\newtheorem{cormain}[thmmain]{Corollary}
\newtheorem{qun}[thm]{Question}

\title{Remarks on monotone Lagrangians in $\CC^n$}
\author{Jonathan David Evans}

\address{ETH Z\"{u}rich, R\"{a}mistrasse 101, 8092 Z\"{u}rich, Schweiz}
\email{jonny.evans@math.ethz.ch}
\author{Jarek K\k{e}dra}
\address{University of Aberdeen and University of Szczecin}
\email{kedra@abdn.ac.uk}
\begin{document}
\begin{abstract}
We derive some restrictions on the topology of a monotone Lagrangian submanifold $L\subset\CC^n$ by making observations about the topology of the moduli space of Maslov 2 holomorphic discs with boundary on $L$ and then using Damian's theorem which gives conditions under which the evaluation map from this moduli space to $L$ has nonzero degree. In particular we prove that an orientable 3-manifold admits a monotone Lagrangian embedding in $\CC^3$ only if it is a product, which is a variation on a theorem of Fukaya. Finally we prove an h-principle for monotone Lagrangian immersions.
\end{abstract}
\maketitle
\section{The results}
This paper is concerned with monotone Lagrangian embeddings into the standard symplectic vector space $\CC^n$. In the next section we will put our results into context.
\begin{dfn}\label{mono}
A Lagrangian $L\rightarrow\CC^n$ is called \emph{monotone} if there exists $K>0$ such that the symplectic area of a disc $D$ with boundary on $L$ is $K\mu(D)$ where $\mu(D)$ is its Maslov index.
\end{dfn}
\begin{thmmain}\label{simpvol}
Suppose $L$ is a closed orientable spin manifold which is a connected
sum of aspherical manifolds of odd dimension. If $L$ admits a monotone Lagrangian
embedding in $\CC^n$ then $L$ has vanishing simplicial volume. In particular all the
aspherical connected summands have vanishing simplicial volume.
\end{thmmain}
Note that `most' aspherical manifolds have nonvanishing simplicial volume, in particular those which admit negative curvature metrics \cite{GroVol} or more generally those with word-hyperbolic fundamental group \cite{Min}. The idea of the proof of Theorem \ref{simpvol} is that if one can `bound the complexity' of a moduli space of discs with boundary on $L$ then $L$ must be `even less complicated'. The second theorem we prove also has this flavour and is a variation on a theorem of Fukaya \cite{Fuk}.
\begin{thmmain}\label{fukimp}
If $L$ is a closed orientable 3-manifold which admits a monotone Lagrangian embedding in $\CC^3$ then $L$ is diffeomorphic to a product $S^1\times\Sigma_g$ where $\Sigma_g$ is a closed orientable surface of genus $g$.
\end{thmmain}

In Section \ref{hprin} we will see that both of these theorems can be understood as giving a lower bound on the number of double points of monotone Lagrangian immersions whose existence is guaranteed by an h-principle which we prove there. We will also show that the h-principle gives us many new examples of monotone Lagrangian embeddings by a stabilisation process observed in the non-monotone case by Audin-Lalonde-Polterovich \cite{ALP}. For example we exhibit monotone Lagrangian embeddings of $S^1\times \Sigma_g$ in $\CC^3$ for any genus $g$, so that Theorem \ref{fukimp} is `sharp'.

\section{Lagrangian embedding problems}
It is an old problem in symplectic topology to determine when a given $n$-manifold $L$ admits a Lagrangian embedding into the standard symplectic vector space $\CC^n$. If one requires only a Lagrangian immersion there is an h-principle due independently to Gromov and Lees \cite{Lees} which guarantees existence as soon as we know that the complexified tangent bundle $TL\otimes_{\RR}\CC$ is trivial. Moreover it classifies regular Lagrangian isotopy classes of immersions by the homotopy class of the Gauss map $L\rightarrow\Lambda_n$ to the Lagrangian Grassmannian. The example to bear in mind is that of a closed orientable 3-manifold, since then the tangent bundle is trivial and there is always a Lagrangian immersion.

The corresponding problem for embeddings is much harder and the first theorem on the subject was due to Gromov
\begin{thm}[Gromov \cite{Gro}]
An exact Lagrangian immersion $f:L\looparrowright\CC^n$ has at least one double point.
\end{thm}
Here a Lagrangian immersion $f$ is called \emph{exact} if $[f^*\lambda]=0\in H^1(L;\RR)$ where $\lambda$ is the Liouville form $\sum_{i=1}^nq_idp_i$ (satisfying $d\lambda=\omega$). When $H^1(L;\RR)=0$ (for example $L\cong S^n$) this means that there can be no Lagrangian embedding.

Note that there is also an h-principle for exact Lagrangian immersions, so Gromov's theorem tells us about the failure of the embedded version of that h-principle. To obtain results about the failure of the original (non-exact) h-principle one has to work harder. The next theorem, due to Fukaya, represents the state of the art in dimension 3.
\begin{thm}[Fukaya \cite{Fuk}]\label{fuk}
If $L$ is a closed orientable prime 3-manifold which admits a Lagrangian embedding in $\CC^3$ then $L$ is diffeomorphic to a product $S^1\times\Sigma$.
\end{thm}
This is a technically very difficult theorem and relies on an extension of the virtual perturbation theory of \cite{FOOO}. A slightly easier theorem, proved using symplectic field theory, gives a more geometric restriction.
\begin{thm}[Viterbo, Eliashberg \cite{SFT}]\label{VitEl}
A closed orientable manifold admitting a metric of negative sectional curvature does not embed as a Lagrangian submanifold of $\CC^n$.
\end{thm}
Note that it is very difficult to say anything about manifolds with nontrivial connected sum decompositions (i.e. non-prime) using the above theorems. Certainly the Viterbo-Eliashberg result relies on the existence of a negatively curved metric and Fukaya's theorem relies on asphericity, both of which are difficult to achieve for connected sums (except maybe for connected sums with exotic spheres \`{a} la Farrell and Jones \cite{FJ}). In particular the following question is wide open.
\begin{qun}
Can a connected sum of closed orientable hyperbolic 3-manifolds be embedded as a Lagrangian submanifold in $\CC^3$?
\end{qun}

In the present note we proceed as Gromov did and make further simplifying assumptions about the Lagrangian. We assume it is monotone in the sense of Definition \ref{mono}. This assumption will enable us to avoid much technical difficulty and will allow us to make deductions akin to Fukaya's theorem in the case where the manifold is not prime. In the last section we will observe that there is an h-principle for monotone Lagrangian immersions, which gives the results an interpretation as the failure of an h-principle: otherwise the assumption of monotonicity seems somewhat artificial.

The main ingredient is the following theorem of Damian which crucially relies on monotonicity of the Lagrangian and employs tools which are significantly easier to set up than \cite{Fuk}.
\begin{thm}[Damian {\cite[Theorem 1.5(b)-(c)]{Dam}}]
Let $M$ be a monotone symplectic manifold which has the property that any compact subset is displaceable through a Hamiltonian isotopy. Let $L\subset M$ be a monotone Lagrangian submanifold. Denote by $\tilde{L}$ the universal cover of $L$. If $L$ is orientable and has the property
\[H_{2i+1}(\tilde{L},\ZZ/2)=0\]
for any integer $i$ then $N_L=2$. Moreover, for any almost complex structure $J$ which is compatible with the symplectic form, $L$ has the property that through every $p\in L$ there is a $J$-holomorphic disk $w\colon(D,\partial D)\to(M, L)$ such that:
\begin{itemize}
\item The Maslov index $\mu(w)$ equals $2$.
\item $p\in w(\partial D)$.
\item $w(\partial D)$ is non zero in $\pi_1(L)$.
\end{itemize}
\end{thm}
In fact Damian shows that for generic $J$ and a fixed generic point $x_0\in L$ there is a homotopy class $A\in\pi_2(M,L,x_0)$ of Maslov 2 discs with boundary on $L$ which contains an odd, finite number of holomorphic discs through $x_0$. See the last paragraph of the proof of Theorem 1.5 on {\cite[Page 453]{Dam}}. For the remainder of this paper we will take $M$ to be the standard symplectic vector space $\CC^n$, so
\[\pi_2(\CC^n,L,x_0)\cong\pi_1(L,x_0).\]
\begin{prp}\label{degree}
Suppose that $L\subset\CC^n$ is oriented and spin. Let $\beta$ denote the free homotopy class of loops containing the class $A\in\pi_1(L,x_0)$ and let $\mM_{0,1}(\beta,J)$ denote the moduli space of $J$-holomorphic discs with one boundary marked point representing the free homotopy class $\beta$. Then the evaluation map $\ev\colon\mM_{0,1}(\beta,J)\to L$ has nonzero degree over the integers.
\end{prp}
The proposition makes sense because, for generic $J$, the moduli space is a compact, smooth manifold of dimension $n=\dim(L)$. Moreover the moduli space is orientable as $L$ is spin \cite{FOOO}, hence we can make sense of its fundamental class over the integers.

\section{Proof of Proposition \ref{degree}}\label{topology}

Observe that for generic $J$ the moduli space is a compact, smooth, orientable manifold of dimension $n$. Let $\mM_{0,0}(\beta,J)$ denote the moduli space of unmarked $J$-holomorphic discs with boundary on $L$ in the class $\beta$ and let $E$ denote the space of parametrised $J$-holomorphic discs with boundary on $L$ in the class $\beta$, so that $\mM_{0,0}(\beta,J)=E/\mathbf{P}SL(2,\RR)$. We have
\[\mM_{0,1}(\beta,J)=E\times_{\mathbf{P}SL(2,\RR)} S^1\]
and forgetful map $\mM_{0,1}(\beta,J)\to\mM_{0,0}(\beta,J)$ is the projection map of this associated bundle.

Since both the total space and the base of this bundle are oriented, the fibres are oriented and so its structure group is $\OP{Diff}^+(S^1)$. We pick a reduction of this structure group to $SO(2)$ using the fact that $\OP{Diff}^+(S^1)\simeq SO(2)$, making $\mM_{0,1}(\beta,J)\to\mM_{0,1}(\beta,J)$ into a principal circle bundle.

Let $y\in\mM_{0,1}(\beta,J)$ be a basepoint. The fibre circle through $y$ represents a central element of the fundamental group. To see this, note that if $S^1\to P\to B$ is a principal circle bundle then $\pi_1(P)$ is generated by lifts $\tilde{x}_i$ of generators $x_i$ of $\pi_1(B)$ and the class $\sigma$ of the fibre. The commutator of $\sigma$ with $\tilde{x}_i$ is $\tilde{x}_i\sigma\tilde{x}_i^{-1}=\sigma$ since this loop is obtained by transporting the fibre around a loop in the base and the monodromies of a principal circle bundle preserve the orientation of the fibre. Therefore $\sigma$ is central.

We will abstract the situation, forgetting about holomorphic discs and Lagrangian submanifolds, and just think of the following setting:
\begin{itemize}
\item $M$ and $N$ are compact, smooth, oriented manifolds of dimension $n$. Moreover we may as well assume that $M$ is connected,
\item $M$ is the total space of a principal $S^1$-bundle and the action of $t\in S^1$ is written $y\mapsto t\cdot y$,
\item if $y\in M$ is a basepoint then the orbit $\{t\cdot y:t\in S^1\}$ represents a central element of $\pi_1(M,y)$.
\item $f\colon M\to N$ is a smooth map.
\end{itemize}
For each $y\in M$ we write $L_y\colon S^1\to N$ for the loop $L_y(t)=f(t\cdot y)$. The relevant property of $f$ coming from Damian's theorem is the existence of a regular value $x_0\in N$ and a homotopy class $A\in\pi_1(N,x_0)$ such that
\begin{equation}\label{keyprop}\#\{y\in f^{-1}(x_0)\colon [L_y]=A\}\equiv 1\mod 2.\end{equation}
\begin{lma}\label{topo}
In this situation, the centraliser of $A$ in $\pi_1(N,x_0)$ is of finite-index.
\end{lma}
\begin{proof}
Suppose it is not and let $\gamma_1,\gamma_2,\ldots$ be an infinite sequence of embedded loops $\gamma_i\colon[0,1]\to N$ such that the classes $A_i=[\gamma_i]\cdot A\cdot[\gamma_i]^{-1}$ are not pairwise homotopic. There exists a ball $U_i$ centred at $x_0$ and an arbitrarily small perturbation $f_i$ of $f$ which agrees with $f$ on $f^{-1}(U_i)$, is transverse to $\gamma_i$ and still satisfies Equation \eqref{keyprop}. The cobordisms
\[W_i=\{(y,t)\in M\times[0,1]\colon f_i(y)=\gamma_i(t)\}\]
are then smooth, compact 1-dimensional manifolds. By transporting along $\gamma_i$ we identify $\pi_1(N,\gamma_i(t))$ with $\pi_1(N,x_0)$. With this, $W_i$ decomposes into a union $\bigcup_{c\in\pi_1(N,x_0)}W_i(c)$ where $(y,t)\in W_i(c)$ if $[L_y]$ transports back along $\gamma_i$ to $c$ based at $x_0$.

By Equation \eqref{keyprop} $W_i(A)\cap M\times\{0\}$ consists of an odd number of points and hence so does $W_i(A)\cap M\times\{1\}$. In particular there exists an embedded interval $\iota_i\colon [0,1]\to W_i$ with $\iota_i(0)\in M\times\{0\}$ and $\iota_i(1)\in M\times\{1\}$. Write $\OP{pr}_M$ for the projection of $M\times[0,1]$ to $M$. If $y_t=\OP{pr}_M(\iota_i(t))$ then we have a free homotopy of loops $L_{y_0}$ to $L_{y_1}$ and the basepoints traverse the loop $f_i\circ\OP{pr}_M\circ\iota_i$ in $N$ which is homotopic to $\gamma_i$. Therefore
\[[L_{y_1}]=[\gamma_i]\cdot[L_{y_0}]\cdot[\gamma_i]^{-1}=A_i\]
But $y_1\in f^{-1}(x_0)$ so there are only finitely many possibilities for $y_1$. This is a contradiction to the claim that $A_i$ runs over infinitely many distinct homotopy classes.
\end{proof}
Consider the centraliser subgroup of $A$, $Z(A)\subset\pi_1(N,x_0)$, and let $k\colon(\bar{N},\bar{x}_0)\to(N,x_0)$ be the based cover such that $k_*\pi_1(\bar{N},\bar{x}_0)=Z(A)$. Let $y\in f^{-1}(x_0)$ be such that $[L_y]=A$. Since the centraliser of the class of the loop $\{t\cdot y\}$ is the whole of $\pi_1(M,y)$, the image $f_*\pi_1(M,y)$ is contained in $Z(A)$ and so the map $f$ lifts to a map $\bar{f}\colon M\to\bar{N}$ taking $y$ to $\bar{x}_0$.
\begin{lma}
The preimage $\bar{f}^{-1}(\bar{x}_0)\subset M$ consists of those $y'\in f^{-1}(x_0)$ such that $[L_y]=A$.
\end{lma}
\begin{proof}
If $y'\in f(x_0)$ satisfies $[L_y]=A$ then any path $\delta$ in $M$ connecting $y$ and $y'$ becomes a loop $f\circ\delta$ in $N$ which centralises $A$. Therefore it lifts to a loop $\bar{f}\circ\delta$ based at $\bar{x}_0$, proving that $y'\in\bar{f}^{-1}(\bar{x}_0)$.

If $y'\in f^{-1}(x_0)$ does not satisfy $[L_y]=A$ and $\delta$ is a path in $M$ connecting $y$ to $y'$ then $f\circ\delta$ is a loop in $M$ which does not centralise $A$. This implies that $\bar{f}\circ\delta$ is a path in $\bar{N}$ connecting $\bar{x}_0$ to some other point in $k^{-1}(x_0)$, so $y'\not\in\bar{f}^{-1}(\bar{x}_0)$.
\end{proof}

Now Equation \ref{keyprop} implies that the degree of $\bar{f}$ is odd, in particular nonzero, so the degree of the composite $k\circ\bar{f}$ is nonzero.

This proves Proposition \ref{degree}.\qed

\section{Proofs of Theorems}

We need a preliminary lemma.

\begin{lma}[Fukaya's trick {\cite[Remark 12.2(A)]{Fuk}}]\label{fuktrick}
The moduli space $\mM_{0,0}(\beta,J)$ admits a finite cover $\pi\colon\widetilde{\mM}_{0,0}(\beta,J)\to\mM_{0,0}(\beta,J)$ such that the pullback bundle
\[
\begin{CD}
\widetilde{\mM}_{0,1}(\beta,J)@>>>\mM_{0,1}(\beta,J)\\
@VVV @VVV\\
\widetilde{\mM}_{0,0}(\beta,J)@>>{\pi}>\mM_{0,0}(\beta,J)
\end{CD}
\]
is diffeomorphic to a product $\widetilde{\mM}_{0,1}(\beta,J)\cong\widetilde{\mM}_{0,0}(\beta,J)\times S^1$. If $n=3$ then we do not need to take a finite cover.
\end{lma}
\begin{proof}
Let $\widetilde{\mM}_{0,0}(\beta,J)$ be the finite cover whose first homology is torsionfree and let $\widetilde{\mM}_{0,1}(\beta,J)$ denote the pullback circle bundle. If $n=3$ then $\widetilde{\mM}_{0,0}(\beta,J)$ is an orientable two-manifold and hence its first homology is already torsionfree.

Consider the Gysin spectral sequence of the circle bundle $\widetilde{\mM}_{0,1}(\beta,J)\rightarrow\widetilde{\mM}_{0,0}(\beta,J)$. By our choice of cover, the universal coefficients theorem implies that
\[H^2(\widetilde{\mM}_{0,0}(\beta,J);\ZZ)\cong\OP{Hom}(H_2(\widetilde{\mM}_{0,0}(\beta,J);\ZZ),\ZZ)\]
and the Euler class of the circle bundle is identified with the $E_2$-differential
\[d_2\colon H_2(\widetilde{\mM}_{0,0}(\beta,J);\ZZ)\to H_1(S^1;\ZZ)\cong\ZZ\]
in the homology Gysin spectral sequence. If the Euler class were nonzero then the homology class of the fibre would be torsion. However, the homology class of the fibre survives as a nontorsion element in the homology of $\widetilde{\mM}_{0,1}(\beta,J)$ because there is a nontrivial cohomology class (the pullback of the Maslov class along the evaluation map) which evaluates nontrivially on it. Therefore the Euler class of this circle bundle is zero and hence it is a product.
\end{proof}

\subsection{Proof of Theorem \ref{simpvol}}
Let $L$ be a Lagrangian in $\CC^n$ satisfying the assumptions of the theorem. Then its universal cover has vanishing homology in odd degrees and hence satisfies the assumptions of Damian's theorem so by Proposition \ref{degree} there is a component $M$ of the moduli space of Maslov 2 discs with boundary on $L$ and one boundary marked point such that the evaluation map $\ev:M\rightarrow L$ has nonzero degree.

Since the moduli space $\mM_{0,1}(\beta,J)$ is an $S^1$-bundle over the unmarked moduli space we see that the simplicial volume of $\mM_{0,1}(\beta,J)$ vanishes (i.e. the fundamental class can be represented by $\RR$-chains with arbitrarily small $\ell^1$-norm). Hence the simplicial volume of $M$ vanishes.

Since $\ev:M\rightarrow L$ has nonzero degree one can pushforward an infimising sequence of $\RR$-chains representing $[M]$ to obtain a sequence of $\RR$-chains representing a fixed multiple of $[L]$ and the $\ell^1$-norm of these still tends to zero. Hence $L$ has vanishing simplicial volume.

Note that simplicial volume is additive under connected sum {\cite[Section 0.2]{GroVol}}, hence we can deduce that all summands have vanishing simplicial volume. \qed

\subsection{Proof of Theorem \ref{fukimp}}
We start with the observation that any orientable Lagrangian in $\CC^3$ has infinite fundamental group. This is because
\[H^1(L;\RR)=\Hom(\pi_1(L);\RR)\]
has nonzero rank by {\cite[Theorem $0.4.A_2$]{Gro}}. Therefore the universal cover $\tilde{L}$ has no homology in odd degrees and hence $L$ satisfies the hypotheses of Damian's theorem. This means that (for generic $J$) there is a free homotopy class $\beta$ of loops in $L$ with Maslov number 2 such that the moduli space $\mM_{0,1}(\beta,J)$ of $J$-holomorphic discs with boundary on $L$ representing $\beta$ and with one boundary marked point is a compact smooth orientable manifold of dimension 3. By Proposition \ref{degree} there is a component $M\subset\mM_{0,1}(\beta,J)$ such that restriction of the the evaluation map
\[\mM_{0,1}(\beta,J)\rightarrow L\]
to $M$ has nonzero degree. We saw in Lemma \ref{fuktrick} that $M$ is diffeomorphic to a product.

The remainder of the proof is similar in spirit to that of Fukaya. Let $K$ be the cover of $L$ corresponding to the subgroup $\ev_*\pi_1(\bar{M})$. The map $\ev$ has nonzero degree but factors through the covering $K\to L$; since degree is multiplicative under composition we see that this covering has finite degree so $K$ is compact. Therefore $L$ has a finite cover $K\rightarrow L$ whose fundamental group is a quotient of $\pi_1(M)\cong\ZZ\times\Gamma$ where $\Gamma=\pi_1(\Sigma)$ and $\Sigma$ is the unmarked moduli space. Since the generator of $\ZZ\times\{0\}$ survives in $\pi_1(L)$ (it has nontrivial Maslov number) the quotient in question has the form $\ZZ\times\Gamma'$.

If this cover were not prime then its fundamental group would admit a decomposition as a free product and hence its centre would be trivial. However the center of $\ZZ\times\Gamma'$ contains $\ZZ\times\{0\}$. Therefore $K$ is prime. Either $K$ is $S^1\times S^2$ or else it is aspherical. Suppose $K$ is aspherical (and therefore irreducible). Since it is a finite cover of a manifold with positive first Betti number it is \emph{sufficiently large} in the sense of Waldhausen \cite{Wal}. By Waldhausen's rigidity theorem \cite{Wal} the classifying map $K\rightarrow K(\ZZ\times\Gamma',1)=S^1\times\Sigma$ is homotopic to a homemorphism. Hence $K$ is diffeomorphic to a product.

In the case when $\Sigma=S^2$ we see that the only orientable 3-manifold covered by $S^1\times S^2$ is $S^1\times S^2$. Hence $L$ is a product.

In the case when $\Sigma$ is aspherical we see by a theorem of Meeks and Scott {\cite[Theorem 2.1]{MS}} that the finite group action preserves some Thurston (product) metric on $S^1\times\Sigma$. In particular, $L$ is the mapping torus of a periodic diffeomorphism (indeed isometry) $\kappa$ of a free quotient $\Sigma'$ of $\Sigma$. We now claim that $\kappa=\id$. To see this, write $\MT(\kappa)$ for the mapping torus and contemplate the commutative diagram for the quotient
\[\begin{CD}
S^1\times\Sigma@>>>\MT(\kappa)\\
@VVV @VV{\pi}V\\
S^1 @>>> S^1
\end{CD}\]
The bottom map $S^1\rightarrow S^1$ has degree equal to the period of $\kappa$. Say this period is $d$. Since $\pi_1(S^1\times\Sigma)=\ZZ\times\Gamma'$ and $(1,\id_{\Gamma'})$ is sent to an element $A\in\pi_1(L)$ representing the free homotopy class $\beta$ we see that $\pi_*A=d[S^1]$. Since the Maslov class of $\beta$ is 2 which is minimal, $d=1$. Hence $\kappa=\id$ and $L$ is a product.\qed

\subsection{Remarks on Theorem \ref{fukimp}}
As Fukaya points out, the conclusion of Theorems \ref{fuk} and \ref{fukimp} does not hold for general Lagrangians (without either the primeness or monotonicity assumption). Any orientable 3-manifold immerses with a finite number of transverse double points as a Lagrangian in $\CC^3$ by the Gromov-Lees h-principle \cite{Lees} and the double points can be removed (orientably) by Lagrangian surgery \cite{Pol}. Therefore any 3-manifold embeds as a Lagrangian after taking the connected sum with $S^1\times S^2$ some finite number of times.
\begin{cormain}
Let $L$ be a closed orientable 3-manifold and $L\looparrowright\CC^3$ a monotone Lagrangian immersion with $k$ double points. Unless $L\cong S^3$ and $k=1$, no Lagrangian surgery of the image resolving all the double points is Lagrangian isotopic to a monotone embedding.
\end{cormain}
\begin{proof}
The Lagrangian surgery procedure affects $L$ by taking the connected sum with $k$ copies of $S^1\times S^2$. The fundamental group of the surgered Lagrangian is then a free product of $\pi_1(L)$ with $k$ copies of $\ZZ$. Unless $\pi_1(L)=\{1\}$ and $k=1$ this cannot be the fundamental group of a product 3-manifold. By the Poincar\'{e}-Perelman theorem this means that $L$ must be a 3-sphere with one double point (indeed such an immersion exists, the Whitney immersion, and it admits a resolution Lagrangian isotopic to a monotone $S^1\times S^2$).
\end{proof}

\section{An h-principle for monotone Lagrangian immersions}\label{hprin}
Often, monotonicity is introduced as a technical tool to make Floer theory work and as such it seems perverse to prove restrictions on the topology of monotone Lagrangians for their own sake. The purpose of this section is to observe that there is an h-principle for monotone Lagrangian immersions. This means that the `topological restrictions' proved above are really measuring the failure of an h-principle for embeddings, which makes them more than just idle curiosities. The h-principle also allows us to construct many interesting monotone Lagrangian embeddings after `stabilising' by taking products with $S^1$.
\subsection{The h-principle}
If $L$ is an $n$-manifold and $L\looparrowright\CC^n$ is a Lagrangian immersion it defines a \emph{Gauss-map} $\widehat{df}:L\rightarrow\Lambda$ to the Lagrangian Grassmannian. More generally, given an isotropic bundle monomorphism $F:TL\rightarrow\CC^n$ we have a Gauss-map $\hat{F}:L\rightarrow\Lambda$. One can pull back the Maslov class $\mu\in H^1(\Lambda,\ZZ)$ along $\hat{F}$ and obtain a class $\mu_F\in H^1(L;\ZZ)$. We call this the Maslov class of $f$ if $F=df$.
\begin{dfn}
Fix a real number $K$. Let $\lambda$ denote the 1-form $\sum q_idp_i$ on $\CC^n$ and $\omega=d\lambda$ the standard symplectic form. We say a Lagrangian immersion
\[f\colon L\looparrowright\CC^n\]
is $K$-monotone if $[f^*\lambda]=\frac{\pi K^2}{2}\mu_{df}$, and that $f$ is monotone if it is $K$-monotone for some $K\geq 0$. Note that since $f$ is Lagrangian $f^*\lambda$ is a closed 1-form so its cohomology class makes sense.\end{dfn}
\begin{thmmain}[h-Principle for $K$-monotone Lagrangian immersions into $\CC^n$]
Fix a number $K$. Suppose we are given a map $f:L\rightarrow \CC^n$ and an isotropic bundle monomorphism
\[
\begin{CD}
TL@>{F}>>T\CC^n\\
@VVV @VVV\\
L@>>{f}>\CC^n
\end{CD}
\]
Then there exists a Lagrangian immersion $\phi:L\rightarrow\CC^n$ which is $\mC^0$-close to $f$ and such that $[\phi^*\lambda]=\frac{\pi K^2}{2}\mu_F$.
\end{thmmain}
\begin{proof}
The proof is an adaptation of the proof of the result in Section 16.3.1 of \cite{EM}, which is itself the $K=0$ case (exact Lagrangians). We prove the case $K>0$ here. Note that the isotropic bundle monomorphism has a natural lift $\tilde{F}$ to an isotropic bundle monomorphism into the contact distribution of the contact manifold $\CC^n\times [0,\pi K^2]/\sim$ (where $z\sim z+\pi K^2$) with contact form $dz-\lambda$ ($z$ is the periodic coordinate), so the problem of finding a $K$-monotone Lagrangian is translated into the problem of finding a Legendrian immersion. Any such Legendrian immersion projects to a $K$-monotone Lagrangian immersion in $\CC^n$.
\end{proof}
\subsection{Applications}\label{h-prin-appl}
We now show how to use the monotone Lagrangian immersions constructed by the above h-principle to construct monotone Lagrangian embeddings in higher dimensions.
\begin{prp}[c.f. {\cite[Proposition 1.2.3]{ALP}}]\label{ALP}
If $\iota_1\colon L_1\looparrowright\CC^n$ is a $K$-monotone Lagrangian immersion and $\iota_2\colon L_2\rightarrow\CC^m$ is a $K$-monotone Lagrangian embedding then there is a $K$-monotone Lagrangian embedding $\iota_3\colon L_3=L_1\times L_2\to\CC^{n+m}$.
\end{prp}
\begin{proof}
We first observe that by Weinstein's neighbourhood theorem there is a neighbourhood $U\subset T^*L_1$ of the zero section in $T^*L_1$ and an immersion $I\colon U\to\CC^n$ extending the map $\iota_1$ (identifying $L_1$ with the zero-section). Let $f\colon U\to\RR^m\subset\CC^m$ be a function such that $(I,f)\colon U\to\CC^{n+m}$ is an embedding. Let $B_r(p)$ be the radius $r$ symplectic ball centred at $p\in\CC^m$ and let $\phi_{r,p}:B_r(0)\rightarrow B_r(p)$ be the translation diffeomorphism. Define the map
\[I_r\colon U\times B_r(0)\to\CC^{n+m},\ I_r(u,x)=(I(u),\phi_{r,f(u)}(x))\]
Since $\RR^m\subset\CC^m$ is Lagrangian it is not hard to check that this map is symplectic and moreover, for small enough $r$, it is a symplectic embedding. However, by rescaling $f$ we can assume $r$ is arbitrarily large, in particular we can assume that $L_2$ is contained in $B_0(r)$. Now consider the Lagrangian $L_1\times L_2\subset U\times B_r(0)$. To see that its image $L_3=I_r(L_1\times L_2)$ is $K$-monotone, observe that $\pi_2(\CC^{n+m},L_3)=\pi_1(L_1)\times\pi_1(L_2)$. If $(\gamma_1,\gamma_2)\in\pi_1(L_1)\times\pi_1(L_2)$ and $\lambda_k$ denotes the Liouville form on $\CC^k$ then
\[(I_r^*\lambda_{m+n})(\gamma_1,\gamma_2)=\lambda_n(\gamma_1)+\lambda_m(\gamma_2)\]
where we use the fact that $f$ lands in $\RR^m\subset\CC^m$ on which $\lambda_m$ vanishes. Similarly one can see that
\[\mu_3(\gamma_1,\gamma_2)=\mu_1(\gamma_1)+\mu_2(\gamma_2)\]
where $\mu_i$ is the Maslov homomorphism on $\pi_1(L_i)$. In fact, one can give a simple formula relating the three Lagrangian Gauss maps $L_i\to\Lambda_i$ (see {\cite[Theorem 2.6]{Fuk}}). Since both factors have the same monotonicity constant $K$ we see that $L_3$ is $K$-monotone.
\end{proof}

Interesting examples of monotone Lagrangian immersions $\iota_1\colon L_1\looparrowright\CC^n$ are provided by the h-principle (for example $L_1$ could be any 3-manifold). Taking $\iota_2\colon S^1\to\CC$ as the standard Lagrangian embedding of the unit circle and applying Proposition \ref{ALP} now gives monotone Lagrangian \emph{embeddings} $\iota_3\colon L\times S^1\to\CC^{n+1}$. Indeed, since one can find Lagrangian immersions of surfaces $\Sigma_g\looparrowright\CC^2$ \cite{Giv} we see that there are monotone Lagrangian embeddings $S^1\times\Sigma_g\to\CC^3$ for all genera and hence Theorem \ref{fukimp} is sharp.
\subsection*{Acknowledgements}
The authors would like to acknowledge helpful discussions with Paul Biran, Octav Cornea and Janko Latschev. Enormous thanks also go to Mihai Damian and an anonymous referee for pointing out the necessity of a result like Lemma \ref{topo} when we tried to define our moduli space using based rather than free homotopy classes. Also, this paper owes a huge intellectual debt to the papers \cite{Dam} and \cite{Fuk} of Damian and Fukaya. Ivan Smith first asked J.E. if the connected sum of two hyperbolic manifolds would embed as a Lagrangian in a uniruled manifold and he still doesn't know the answer. J.E. is supported by an ETH Postdoctoral Fellowship. Our collaboration is supported by EPSRC grant EP/I036044/1.

\end{document}